\documentclass[12pt]{amsart}
\usepackage{latexsym}
\usepackage{amssymb,amsmath}
\usepackage{amsfonts}
\usepackage[colorlinks]{hyperref}

\newtheorem{theorem}{Theorem}[section]
\newtheorem{lemma}[theorem]{Lemma}
\newtheorem{proposition}[theorem]{Proposition}
\newtheorem{corollary}[theorem]{Corollary}

\theoremstyle{definition}

\newtheorem{question}[theorem]{Question}

\theoremstyle{remark}
\newtheorem*{remark}{Remark}

\numberwithin{equation}{section}

\newcommand{\GL}{{\mathrm {GL}}}

\newcommand{\SL}{{\mathrm {SL}}}
\newcommand{\PSL}{{\mathrm {PSL}}}

\newcommand{\GU}{{\mathrm {GU}}}

\newcommand{\SU}{{\mathrm {SU}}}
\newcommand{\PSU}{{\mathrm {PSU}}}

\newcommand{\SO}{{\mathrm {SO}}}

\newcommand{\PSp}{{\mathrm {PSp}}}

\newcommand{\Aut}{{\mathrm {Aut}}}
\newcommand{\Out}{{\mathrm {Out}}}

\newcommand{\GF}{{\mathrm {GF}}}

\newcommand{\ord}{{\mathrm {o}}}

\newcommand{\Soc}{{\mathrm {Soc}}}

\newcommand{\lcm}{{\mathrm {lcm}}}
\newcommand{\Gal}{{\mathrm {Gal}}}
\newcommand{\Sz}{{\mathrm {Sz}}}

\newcommand{\ZZ}{{\mathbb Z}}
\newcommand{\NN}{{\mathbb N}}

\newcommand{\ta}{\hspace{0.5mm}^{2}\hspace*{-0.2mm}}

\newcommand{\Centralizer}{\mathbf{C}}

\newcommand{\bC}{{\mathbf{C}}}
\newcommand{\bO}{{\mathbf{O}}}
\newcommand{\bF}{{\mathbf{F}}}

\newcommand{\Al}{\textup{\textsf{A}}}
\newcommand{\Sy}{\textup{\textsf{S}}}

\begin{document}

\title[Variations of Landau's theorem]
{Variations of Landau's theorem for $p$-regular\\ and $p$-singular
conjugacy classes}

\author{Alexander Moret\'{o}}
\address{Departament d'\`{A}lgebra, Universitat de Val\`{e}ncia, 46100 Burjassot, Val\`{e}ncia, Spain} \email{alexander.moreto@uv.es}

\author{Hung Ngoc Nguyen}
\address{Department of Mathematics, The University of Akron, Akron,
Ohio 44325, USA} \email{hungnguyen@uakron.edu}

\thanks{The research of the first author was supported by the  Spanish
Ministerio de Ciencia y Tecnolog\'{\i}a, grant MTM2010-15296 and
PROMETEO/Generalitat Valenciana. The second author is partially
supported by the NSA Young Investigator Grant \#H98230-14-1-0293
    and a BCAS Faculty Scholarship Award from the Buchtel College of Arts and Sciences, The University of Akron.}



\date{\today}

\begin{abstract} The well-known Landau's theorem states that, for any
positive integer $k$, there are finitely many isomorphism classes of
finite groups with exactly $k$ (conjugacy) classes. We study
variations of this theorem for $p$-regular classes as well as
$p$-singular classes. We prove several results showing that the
structure of a finite group is strongly restricted by the number of
$p$-regular classes or the number of $p$-singular classes of the
group. In particular, if $G$ is a finite group with $\bO_p(G)=1$
then $|G/\bF(G)|_{p'}$ is bounded in terms of the number of
$p$-regular classes of $G$. However, it is not possible to prove
that there are finitely many groups with no nontrivial normal
$p$-subgroup and $k$ $p$-regular classes without solving some
extremely difficult number-theoretic problems (for instance, we
would need to show that the number of Fermat primes is finite).
\end{abstract}

\maketitle

%%%%%%%%%%%%%%%%%%%%%%%%%%%%%%%%%%%%%%%%%%%%%%%%%%%%%%%%%%%%%%%%%%%%%%

\section{Introduction and statement of results}
The well-known Landau's theorem~\cite{Landau} states that, for any
positive integer $k$, there are finitely many isomorphism classes of
finite groups with exactly $k$ conjugacy classes or $k$ (ordinary)
irreducible characters. This theorem is indeed an immediate
consequence of the number-theoretic fact proved by Landau himself in
the same paper that the equation $x_1^{-1}+x_2^{-1}+\cdots
+x_k^{-1}=1$ has only finitely many positive integer solutions.
Using the classification of finite simple groups, L.~H\'{e}thelyi
and B.~K\"{u}lshammer~\cite{Hethelyi-Kulshammer} improved this
theorem by replacing all conjugacy classes by only conjugacy classes
of elements of prime power order. A block version of Landau's
theorem was proposed by R.~Brauer in~\cite{Brauer} -- given a
positive integer $k$, are there only finitely many isomorphism
classes of groups which can occur as defect groups of blocks of
finite groups with exactly $k$ irreducible characters? This was
answered affirmatively by K\"{u}lshammer in~\cite{Kulshammer1} for
solvable groups and then in~\cite{Kulshammer2} for $p$-solvable
groups. As far as we know, Brauer's question is still open in
general and indeed, a positive answer to this question follows if
the Alperin-McKay conjecture is correct, as pointed out by
B.~K\"{u}lshammer and G.\,R.~Robinson~\cite{Kulshammer-Robinson}.

Let $G$ be a finite group. Landau's theorem can be restated in the
following form: the order of $G$ is bounded in terms of the class
number of $G$. We remark that the number of classes of $G$ is equal
to the sum of the multiplicities of class sizes as well as the sum
of the multiplicities of character degrees of $G$.
A.~Moret\'{o}~\cite{Moreto} and D.\,A.~Craven~\cite{Craven} have
strengthened Landau's theorem by showing that the order of a finite
group is bounded in terms of the largest multiplicity of its
character degrees. The dual statement for conjugacy classes was
proved only for solvable groups by A.~Jaikin-Zapirain~\cite{Jaikin}
and seems very difficult to prove for arbitrary groups,
see~\cite{Nguyen}.

In this paper, we study a modular version of Landau's theorem. Let
$p$ be a prime. Instead of considering all conjugacy classes, we
consider only $p$-regular classes which are classes of elements of
order not divisible by $p$. The number of $p$-regular classes of $G$
turns out to be equal to the number of irreducible representations
of $G$ over an algebraically closed field of characteristic $p$. We
also study another variation for $p$-singular classes, which are
classes of elements of order divisible by $p$.

Let $k_{p'}(G)$ denote the number of $p$-regular classes of $G$. Of
course, $k_{p'}(G)$ does not say anything about $\bO_p(G)$, so it
makes sense to assume that $\bO_p(G)=1$. The question is: can we
bound $|G|$ in terms of $k_{p'}(G)$? (Or, equivalently, we might
have asked whether $|G/\bO_p(G)|$ is bounded in terms of
$k_{p'}(G)$, without assuming that $\bO_p(G)=1$.) It is easy to see
that obtaining an affirmative answer to this question is very hard.
Let $q=2^n+1$ be a Fermat prime. Consider the Frobenius group
$F_n:=C_q\rtimes C_{q-1}$ and let $\bF(G)$ denote the Fitting
subgroup of $G$. We see that $k_{2'}(F_n)=2$ and
$|F_n/\bF(F_n)|=q-1$. This example shows that if $|G/\bF(G)|$ is
bounded in terms of $k_{2'}(G)$ then there would be finitely many
Fermat primes. Therefore, it does not even seem feasible to bound
$|G/\bF(G)|$. Of course, one could ask for a counterexample to our
question that does not depend on the existence on infinitely many
Fermat primes but, as we will discuss in
Section~\ref{section-example}, our question seems equivalent to some
open (and extremely complicated) questions on prime numbers.

Let $\bO_\infty(G)$ denote the largest normal solvable subgroup of
$G$, which is also referred to as the solvable radical of $G$. Our
first result is the following.

\begin{theorem}\label{main theorem} Let $p$ be a fixed prime and $G$ a
finite group with $\bO_p(G)=1$. If $k=k_{p'}(G)$ denotes the number
of $p$-regular classes of $G$, then $|G/\bO_\infty(G)|$ is
$k$-bounded and $\bO_\infty(G)/\bF(G)$ is metabelian by $k$-bounded.
\end{theorem}

Here we say that $|G/\bO_\infty(G)|$ is \emph{$k$-bounded} or
\emph{bounded in terms of $k$} to mean $|G/\bO_\infty(G)|<f(k)$ for
some real-valued function $f$ on $\NN$. The conclusion that
$\bO_\infty(G)/\bF(G)$ is metabelian by $k$-bounded can be stated in
other words, namely there exists $\bF(G)\subseteq
N\trianglelefteq\bO_\infty(G)$ such that $N/\bF(G)$ is metabelian
and $|\bO_\infty(G)/N|$ is bounded in terms of $k$. Two remarks
regarding Theorem~\ref{main theorem} are in order. Firstly, although
it does not say anything about $\bF(G)$, one easily sees that the
nilpotency class of $\bF(G)$ and even the number of chief factors of
$G$ up to $\bF(G)$ is at most $k_{p'}(G)$. Therefore, the number of
$p$-regular classes of $G$ somehow also controls the structure of
$\bF(G)$. Secondly, we will see in Section~\ref{section-example}
that there is not much room for improvement in Theorem~\ref{main
theorem}.

Though the problem of bounding $|G/\bF(G)|$ is out of reach due to
some difficult number-theoretic questions, we are able to prove that
the $p'$-part of $|G/\bF(G)|$ is indeed bounded in terms of the
number of $p$-regular classes of $G$.

\begin{theorem}\label{main theorem-p'-part} Let $p$ be a fixed prime and $G$ a
finite group with $\bO_p(G)=1$. Then $|G/\bF(G)|_{p'}$ is bounded in
terms of $k_{p'}(G)$.
\end{theorem}

The classification of groups $G$ with $\bO_p(G)=1$ according to
$k_{p'}(G)$ has been considered in several papers. First, with a
block-theoretic motivation, Y.~Ninomiya and T.~Wada classified
in~\cite{nw} the groups with $k_{p'}(G)=2$. Ninomiya then obtained
the classification when $k_{p'}(G)=3$ in~\cite{n3,n1,n2} and
G.~Tiedt classified the solvable groups with $k_{p'}(G)=4$ for $p$
odd in~\cite{tie}. We have classified the groups with trivial
solvable radical and $k_{p'}(G)=4$.

\begin{theorem}\label{B}
Let $p$ be a prime and $G$ a finite group with trivial solvable
radical. Then $k_{p'}(G)=4$ if and only if one of the following
holds:
\begin{enumerate}
\item[(i)] $p=2$ and $G\cong \Al_5$, $\PSL_2(7)$, $\Al_6\cdot 2_1$, $\Al_6\cdot
2_2$, $\PSU_3(3)\cdot 2$, $\PSL_3(4)\cdot 2_1$, $\PSL_3(4)\cdot
2^2$, or $\ta F_4(2)'\cdot 2$.
\item[(ii)] $p=3$ and $G\cong \Al_5$.
\item[(iii)] $p=7$ and $G\cong \PSL_2(7)$.
\end{enumerate}
\end{theorem}

\noindent We conjecture that these are all the nonsolvable groups
with four $p$-regular classes and $\bO_p(G)=1$ but, at this time, we
have not been able to prove this.

Motivated by the results on $p$-regular classes, we have considered
the analogous question for $p$-singular classes. So let $k_p(G)$
denote the number of conjugacy classes of $p$-singular elements in
$G$. We have obtained the following.

\begin{theorem}\label{main theorem for p-singular classes} Let $p$ be a fixed prime, $F$ a positive real number, and $G$ a finite group whose non-abelian
composition factors have order divisible by $p$ and $\bO_{p'}(G)=1$.
Assume that $G$ does not contain $\PSL_2(p^f)$ with $f> F$ as a
composition factor. Assume furthermore that
\begin{enumerate}
\item[(i)] if $p=2$ then the Suzuki groups $\Sz(2^f)$ with $f>F$ are
not composition factors of $G$; and
\item[(ii)] if $p=3$ then the Ree groups $\ta G_2(3^f)$ with $f>F$ are
not composition factors of $G$.
\end{enumerate}
Then $|G/\bO_\infty(G)|$ is $(F,k_p(G))$-bounded and
$\bO_\infty(G)/\bF(G)$ is metabelian by $(F,k_p(G))$-bounded.
\end{theorem}

As before, we say that $|G/\bO_\infty(G)|$ is
\emph{$(F,k_p(G))$-bounded} or \emph{bounded in terms of $F$ and
$k_p(G)$} to mean that $|G/\bO_\infty(G)|<f(F,k_p(G))$ for some
real-valued function $f$ on $\NN \times \NN$.

The next result is an analog of Theorem~\ref{main theorem-p'-part}.

\begin{theorem}\label{main theorem-p-part} Assume the hypotheses of
Theorem~\ref{main theorem for p-singular classes}. Then
$|G/\bF(G)|_p$ is bounded in terms of $F$ and $k_p(G)$.
\end{theorem}

All the assumptions on the simple groups $\PSL_2(p^f)$,
$\Sz(2^{f})$, $\ta G_2(3^f)$ are necessary. The simple linear group
$\PSL_2(2^f)$ with $f\geq 2$ has a unique class of $2$-singular
elements and indeed this is the single class of involutions. Also,
the Suzuki group $\Sz(2^{f})$ with $f\geq 3$ odd has exactly three
$2$-singular classes, two of which are (unipotent) classes of
elements of order 4 and the other is the unique class of
involutions, see~\cite{Suzuki}.

On the other hand, when $p$ is an odd prime, the simple linear group
$\PSL_2(p^f)$ with $f\geq 1$ has exactly $2(p^{2f}-1)$ $p$-singular
elements and two classes of $p$-singular elements. The Ree group
$\ta G_2(3^f)$ with $f\geq 3$ odd has exactly eight classes of
$3$-singular elements (three classes of order $3$ elements, three of
order 9, and two of order 6), see~\cite{Ree}. Finally, we remark
that the Suzuki groups $\Sz(2^{f})$ with $f\geq 3$ odd are the only
simple groups whose orders are not divisible by $3$.

One could ask whether we can weaken the hypothesis by considering
only classes of $p$-elements. Indeed, our proof shows that if $G$ is
solvable  and $\bO_{p'}(G)=1$ then $G$ has a normal subgroup $N$
such that $N/\bF(G)$ is metabelian and $|G/N|$ is bounded in terms
of the number of classes of $p$-elements. Also, in~\cite{mst} it was
proved that the derived length of a Sylow $p$-subgroup of a finite
group is bounded above by (a linear function of) the number of
classes of $p$-elements. However, our results fail if we replace
``$p$-singular classes'' by ``classes of $p$-elements''. For
instance, the linear group $\PSL_3(3^a)$ has three classes of
$3$-elements for every $a\in\ZZ^+$ (see~\cite{Simpson-Frame}) and
one could find counterexamples with any family of simple groups of
Lie type.

Since the order of every non-abelian simple group is even, the
following is an immediate consequence of Theorems~\ref{main theorem
for p-singular classes} and~\ref{main theorem-p-part}.

\begin{corollary}\label{main theorem for 2-singular classes} Let $F$ be a positive real number, $G$ a finite group with $\bO_{2'}(G)=1$ and assume that $G$ does
not have composition factors isomorphic to $\PSL_2(2^f)$ or
$\Sz(2^{f})$ with $f> F$. Then
\begin{enumerate}
\item[(i)] $|G/\bO_\infty(G)|$ and $|G/\bF(G)|_2$ are both
$(F,k_2(G))$-bounded; and
\item[(ii)] $|\bO_\infty(G)/\bF(G)|$ is metabelian by
$(F,k_2(G))$-bounded.
\end{enumerate}
\end{corollary}

The proofs of Theorems~\ref{main theorem} and~\ref{main theorem for
p-singular classes} are divided into two parts. First we prove that
$|G/\bO_{\infty}(G)|$ is $k_{p'}(G)$-bounded (respectively,
$(F,k_p(G))$-bounded) and then we prove our results in full. These
theorems are of a qualitative nature only. One could obtain explicit
bounds by following the proofs, but these bounds would be far from
best possible. For instance, following the proof of
Theorem~\ref{main theorem} we obtain that if $\bO_{\infty}(G)=1$ and
$k_{p'}(G)=4$ then $|G|\leq (11!^4/16)!$ while from Theorem~\ref{B}
we see that indeed $|G|\leq 35\,942\,400$.

The outline of the paper is as follows. In the next section, we
prove some preliminary lemmas which will be needed later in the
proofs of the main results. In
Section~\ref{section-regular-class-Lietype-groups}, we present some
lower bounds for the number of $p$-regular classes of finite simple
groups of Lie type. These bounds are crucial in the proof of the
first part of Theorem~\ref{main theorem} in Section~\ref{section4}.
Similarly, we state and prove lower bounds for the number of
$p$-singular classes of simple groups that we need for
Theorem~\ref{main theorem for p-singular classes} in
Section~\ref{section-singular-class-simple-groups} and complete the
proof of the first part of Theorem~\ref{main theorem for p-singular
classes} in Section~\ref{section6}. Next, we prove the full versions
of Theorems~\ref{main theorem} and~\ref{main theorem for p-singular
classes} in Section~\ref{section-proofs of maintheorems}. Proofs of
Theorems~\ref{main theorem-p'-part} and~\ref{main theorem-p-part}
have the same flavor and are carried out in
Section~\ref{section-proof-Theorem-p'part}. We discuss the
difficulties of improving on these theorems and some open questions
in Section~\ref{section-example}. Finally we prove Theorem~\ref{B}
in Section~\ref{section-proof-Theorem-B}.

\begin{remark} After this paper was completed, we learned that part
of Theorem~\ref{main theorem} was already obtained by D.\,S.~Passman
in~\cite{Passman}. He raised the following problem: \emph{Let $G$
have precisely $n$ irreducible representations over an algebraically
closed field of characteristic $p>0$. How much of the structure of
$G$ can be bounded by a function of $n$, possibly depending upon
$p$}? Using a result of R.\,M.~Guralnick
\cite[Theorem~3.4]{Passman}, Passman proved that
$|G/\bO^{p'}(\bO_\infty(G))|$ is bounded in terms of $n$, where
$\bO^{p'}(\bO_\infty(G))$ is the smallest normal subgroup of
$\bO_\infty(G)$ such that $\bO_\infty(G)/\bO^{p'}(\bO_\infty(G))$ is
a $p'$-group (see \cite[Corollary~3.5]{Passman}). The first part of
Theorem~\ref{main theorem} follows from Passman's result while the
second part and Theorem~\ref{main theorem-p'-part} partly solve
Passman's problem. As the arguments are different and the proof is
just sketched in Passman's paper, for completeness we have decided
to keep our proof of the first part of Theorem~\ref{main theorem}.
\end{remark}

%%%%%%%%%%%%%%%%%%%%%%%%%%%%%%%%%%%%%%%%%%%%%%%%%%%%%%%%%%%%%%%%%%%%%%%%%%%%%%%%%%%%

\section{Preliminaries}\label{section-preliminaries}
In this section we prove some easy but useful lemmas. The first two
of them are well-known, but for the reader's convenience we include
their short proofs. We will write $\ord(g)$ to denote the order of a
group element $g$.

\begin{lemma}\label{lemma k2(G/N) <k2(G)} Let $G$ be a finite group and $N\trianglelefteq G$. Then $k_{p'}(G/N)\leq k_{p'}(G)$ and $k_p(G/N)\leq
k_p(G)$.
\end{lemma}

\begin{proof} The first part of the lemma is obvious as $k_{p'}(G)$ is exactly the number
of irreducible $p$-Brauer characters of $G$ and every $p$-Brauer
character of $G/N$ can be viewed as a $p$-Brauer character of $G$.
So it remains to prove the second part of the lemma. Let
$k=k_p(G/N)$ and let $\{g_1N,g_2N,...,g_kN\}$ be a collection of
representatives for $p$-singular classes of $G/N$. Since $\ord(gN)$
divides $\ord(g)$ for every $g\in G$, we see that $g_1,g_2,...,g_k$
are $p$-singular. It is obvious that $g_1,g_2,...,g_k$ are in
different classes of $G$ and hence the lemma follows.
\end{proof}

\begin{lemma}\label{normal}
Let $G$ be a finite group and $N\trianglelefteq G$. Then
$k_{p'}(N)\leq |G:N|k_{p'}(G)$ and $k_{p}(N)\leq |G:N|k_{p}(G)$.
\end{lemma}

\begin{proof} Consider the conjugation action of $G$ on the set of
$p$-singular classes of $N$. Clearly every orbit of this action has
size at most $|G:N|$. Therefore the number of $p$-singular classes
of $G$ inside $N$ is at least $k_p(N)/|G:N|$ and it follows that
$k_{p}(N)\leq |G:N|k_{p}(G)$. The other inequality is proved
similarly.
\end{proof}

The following lemma is a consequence of a result of S.\,M.
Seager~\cite{sea}.

\begin{lemma}
\label{kel} Let $G$ be a solvable group and $V$ a  faithful
irreducible $G$-module with $|V|=r^a$, where $r$ is a prime power. Assume that the number of $G$-orbits in $V$
is $k$. Then one of the following holds:
\begin{enumerate}
\item
$|G|$ is $k$-bounded.
\item
$G$ is isomorphic to a subgroup of the wreath product of the affine semilinear group $\Gamma(r^a)$
and a symmetric group $\Sy_l$ for some $l$ that is
$k$-bounded and $V=GF(r^a)^l$.
\end{enumerate}
\end{lemma}

\begin{proof}
This follows immediately from Theorem~2.1 of~\cite{kel}, which is a
reformulation of Theorem~1 of~\cite{sea}.
\end{proof}

\begin{corollary}\label{c}
Assume the hypotheses of Lemma~\ref{kel}. Then $G$ is metabelian by
$k$-bounded.
\end{corollary}

\begin{proof} This follows instantly from Lemma~\ref{kel}.
\end{proof}

We end this section with a number-theoretic lemma that will be
needed in the proof of Theorem~\ref{main theorem-p'-part}.

\begin{lemma}\label{lemma-numbertheory} Let $p$ be a fixed prime and $r>1$ be an integer. (Note that $r$ does not need to be
fixed.) Let $n_p$ and $n_{p'}$ respectively denote the $p$-part and
$p'$-part of a positive integer $n$. Then
\[\frac{(r^a-1)_{p'}}{a\cdot a_p}\rightarrow \infty \text{ as }
a\rightarrow \infty.\]
\end{lemma}

\begin{proof} Let $a=p^cb$ with $p\nmid b$. Since
$\gcd(r^{p^c}-1,r^{b}-1)=r-1$ and $r^a-1$ is divisible by both
$r^{p^c}-1$ and $r^{b}-1$, we have
\[(r^a-1)_{p'}\geq \frac{(r^{p^c}-1)_{p'}(r^{b}-1)_{p'}}{(r-1)_{p'}}.\]
It follows that
\[\frac{(r^a-1)_{p'}}{a\cdot a_p}\geq \frac{(r^{p^c}-1)_{p'}}{p^{{2c}}}\cdot \frac{(r^{b}-1)_{p'}}{b(r-1)_{p'}}.\]

By Feit's theorem~\cite{Feit,Roitman} on large Zsigmondy primes, for
each positive integer $N$, there exists a Zsigmondy prime $q$ such
that $(r^{b}-1)_q>bN+1$ for all but finitely many pairs of integers
$(b,r)$ with $b>2$ and $r>1$. We note that when $b$ is large enough
then $q\neq p$. Therefore \[(r^{b}-1)_{p'}>(bN+1)(r-1)_{p'},\] which
implies that
\[\frac{(r^{b}-1)_{p'}}{b(r-1)_{p'}} \rightarrow \infty \text{ as } b \rightarrow\infty.\]

It remains to prove that
\[\frac{(r^{p^c}-1)_{p'}}{p^{{2c}}}\rightarrow \infty \text{ as } c\rightarrow \infty.\]
First we assume that $p\nmid (r-1)$. Since $r^{p^c}\equiv r~(\bmod
p)$, we have $p\nmid (r^{p^c}-1)$ so that
$(r^{p^c}-1)_{p'}=r^{p^c}-1$ and therefore we are done in this case.
So let us consider the case $p \mid (r-1)$.
Using~\cite[Lemma~8]{Lewis-Riedl}, we see that
$(r^{p^c}-1)_p=p^c(r-1)_p$ if $p>2$ or $(r-1)_p>p$ and
$(r^{p^c}-1)_p=p^c(r+1)_p$ otherwise. In any case, we always have
$(r^{p^c}-1)_p\leq p^c(r+1)$. It follows that
\[\frac{(r^{p^c}-1)_{p'}}{p^{{2c}}} \geq
\frac{r^{p^c}-1}{p^{3c}(r+1)}.\] It is now easy to see that
$(r^{p^c}-1)_{p'}/{p^{{2c}}}\rightarrow \infty \text{ as }
c\rightarrow \infty$, as desired.
\end{proof}

%%%%%%%%%%%%%%%%%%%%%%%%%%%%%%%%%%%%%%%%%%%%%%%%%%%%%%%%%%%%%%%%%%%%%%%%%%%%%

\section{$p$-Regular classes of the simple groups of Lie
type}\label{section-regular-class-Lietype-groups}

In this section, we prove Theorem~\ref{main theorem} for simple
groups of Lie type. The next three lemmas provide a lower bound for
the number of $p$-regular classes of a simple group of Lie type.
These bounds are probably not the best possible but they are enough
for our purpose. We basically make use of a result of L.~Babai,
P.\,P.~P\'{a}lfy, and J.~Saxl~\cite{Babai-Palfy-Saxl} on the
proportion of $p$-regular elements in finite simple groups and a
recent result of J.~Fulman and R.\,M.~Guralnick on centralizer sizes
in finite classical groups, see
\cite{Fulman-Guralnick,Burkett-Nguyen}. The detailed structure of
the centralizers of semisimple elements in finite classical groups
can be found in \cite{Hu}.

\begin{lemma}\label{lemma for linear and unitary groups} Let $G=\PSL_n(q)$ or $\PSU_n(q)$ be simple. For any prime $p$, we have
$$k_{p'}(G)>\frac{q^{n-1}}{6n^3}.$$
\end{lemma}

\begin{proof} A lower bound for the smallest centralizer size in finite classical groups is given
in~\cite[\S6]{Fulman-Guralnick}. In particular, the centralizer size
of an element of $\GL_n(q)$ is at least
\[\frac{q^{n-1}(q-1)}{e(1+\log_q(n+1))}.\] Therefore,
for every $x\in\PSL_n(q)$, we have
\begin{align*}|\bC_{\PSL_n(q)}(x)|&\geq
\frac{q^{n-1}(q-1)}{e(1+\log_q(n+1))(q-1)(q-1,n)}\\&=\frac{q^{n-1}}{e(1+\log_q(n+1))(q-1,n)}.\end{align*}

On the other hand, by~\cite[Theorem~1.1]{Babai-Palfy-Saxl}, if $G$
is a simple classical group acting naturally on a projective space
of dimension $m-1$, then the proportion of $p$-regular elements of
$G$ is at least $1/(2m)$. Therefore the number of $p$-regular
elements in $\PSL_n(q)$ is at least
\[\frac{1}{2n}|\PSL_n(q)|.\]
We thus obtain a lower bound on the number of classes of $p$-regular
elements in $\PSL_n(q)$:
\[k_{p'}(\PSL_n(q))\geq\frac{q^{n-1}}{2ne(1+\log_q(n+1))(q-1,n)},\] and the lemma follows for the linear groups.

The arguments for the unitary groups follow similarly by using the
fact that the centralizer size of an element of $\GU_n(q)$ is at
least $q^n(\frac{1-1/q^2}{e(2+\log_q(n+1))})^{1/2}$.
\end{proof}

\begin{lemma}\label{lemma for symplectic and orthogoanl groups} Let $G=\PSp_{2n}(q)$, $\Omega_{2n+1}(q)$, or $\mathrm{P}\Omega_{2n}^\pm(q)$ be simple. For any prime $p$, we have
$$k_{p'}(G)>\frac{q^{n}}{120n^2}.$$
\end{lemma}

\begin{proof} The proof uses the same ideas as the proof of the previous
lemma. We present here a proof for the case
$G=\mathrm{P}\Omega_{2n}^\pm(q)$ only.

By~\cite[Theorem~6.13]{Fulman-Guralnick}, the centralizer size of an
element of $\SO^\pm_{2n}(q)$ is at least
\[q^n\left(\frac{1-1/q}{2e(\log_q(4n)+4)}\right)^{1/2}.\] Therefore, for every
$x\in\mathrm{P}\Omega_{2n}^\pm(q)$ we have
\[|\bC_{\mathrm{P}\Omega_{2n}^\pm(q)}(x)|\geq\frac{q^n}{4}\left(\frac{1-1/q}{2e(\log_q(4n)+4)}\right)^{1/2}.\]
As the number of $p$-regular elements in
$\mathrm{P}\Omega_{2n}^\pm(q)$ is at least
$|\mathrm{P}\Omega_{2n}^\pm(q)|/(4n)$, it follows that
\[k_{p'}(\mathrm{P}\Omega_{2n}^\pm(q))\geq \frac{q^n}{16n}\left(\frac{1-1/q}{2e(\log_q(4n)+4)}\right)^{1/2}\]
and the lemma follows for the simple orthogonal groups in even dimension.
\end{proof}

\begin{lemma}\label{lemma for exceptional groups} Let $G_r(q)$ be a finite simple exceptional
group of Lie type of rank $r$ defined over a field of $q$ elements.
Then, for any prime $p$,
\[k_{p'}(G_r(q))>cq^r,\] where $c$ is a universal constant not
depending on $G_r(q)$ and $p$.
\end{lemma}

\begin{proof} By~\cite[Theorem~6.15]{Fulman-Guralnick}, for every $x\in G_r(q)$, we
have
\[|\Centralizer_{G_r(q)}(x)|\geq
\frac{q^r}{A(\min\{q,r\})(1+\log_qr)}\geq\frac{q^r}{32A},\] where
$A$ is an absolute constant not depending on the group $G_r(q)$. For
a prime $p$, it is shown in~\cite[Theorem~1.1]{Babai-Palfy-Saxl}
that the proportion of $p$-regular elements in a simple exceptional
group of Lie type is greater than $1/15$. We deduce that
\[k_{p'}(G_r(q))>\frac{1}{15}\cdot \frac{q^r}{32A}= \frac{q^r}{480A}.\]
The lemma now follows with $c:=1/(480A)$.
\end{proof}

Lemmas~\ref{lemma for linear and unitary groups}, \ref{lemma for
symplectic and orthogoanl groups}, and \ref{lemma for exceptional
groups} imply that a simple group of Lie type has many classes of
$p$-regular elements when its order is large. In order to prove
Theorem~\ref{main theorem}, we need a bit more than that.

\begin{proposition}\label{Out theorem} Let $G$ be a simple group of Lie
type and let $k^\ast_{p'}(G)$ denote the number of $\Aut(G)$-classes
of $p$-regular elements in $G$. Then we have
\[k^\ast_{p'}(G)\rightarrow \infty \text{ as } |G|\rightarrow \infty.\]
\end{proposition}

\begin{proof} We remark that $k^\ast_{p'}(G)\geq
k_{p'}(G)/|\Out(G)|$ where $\Out(G)$ is the outer automorphism group
of $G$. First we assume that $G=G_r(q)$ is a finite simple
exceptional group of Lie type of rank $r$ defined over a field of
$q=\ell^f$ elements, where $\ell$ is prime. Then it is known that
$|\Out(G)|\leq 6f$ (see~\cite[p.~58]{Gorenstein} for instance). It
follows by Lemma~\ref{lemma for exceptional groups} that
\[k^\ast_{p'}(G)>\frac{cq^r}{6f}.\] It is now easy to see that $k^\ast_{p'}(G)\rightarrow \infty$ as $q\rightarrow\infty$.

The simple classical groups are treated similarly. We will prove
only the case of linear groups as an example. So assume that
$G=\PSL_n(q)$ where $q=\ell^f$. Then $|\Out(G)|=(2,q-1)f$ when $n=2$
and $|\Out(G)|=2(n,q-1)f$ when $n>2$. In particular, $|\Out(G)|\leq
2nf$. Therefore, by Lemma~\ref{lemma for linear and unitary groups},
we have
\[k^\ast_{p'}(G)>\frac{q^{n-1}}{12fn^4}.\]
Again, one can easily show that $q^{n-1}/(12fn^4)\rightarrow\infty$
as either $q\rightarrow \infty$ or $n\rightarrow\infty$. The proof
is complete.
\end{proof}

%%%%%%%%%%%%%%%%%%%%%%%%%%%%%%%%%%%%%%%%%%%%%%%%%%%%%%%%%%%%%%%%%%%%%%%%%%%%%%%%%%%%

\section{First part of Theorem~\ref{main theorem}}\label{section4}

The goal of this section is to prove the following result.

\begin{theorem}\label{main1}
Let $G$ be a finite group. Then $|G/\bO_{\infty}(G)|$ is
$k_{p'}(G)$-bounded.
\end{theorem}

\begin{proof} By Lemma~\ref{lemma k2(G/N) <k2(G)}, we have $k_{p'}(G/\bO_\infty(G))\leq k_{p'}(G)$. Therefore, in order to
prove the theorem, it suffices to assume that $\bO_\infty(G)$ is
trivial. We now need to show that $|G|$ is bounded in terms of
$k_{p'}(G)$.

Let $\Soc(G)$ denote the socle of $G$, which in this case coincides
with the generalized Fitting subgroup of $G$. Since the generalized
Fitting subgroup contains its own centralizer, we have
$\Centralizer_G(\Soc(G))\subseteq \Soc(G)$. It follows that $G$ is
embedded in $\Aut(\Soc(G))$, the automorphism group of $\Soc(G)$. As
$|\Aut(\Soc(G))|$ is bounded in terms of $|\Soc(G)|$, it is enough
to bound $|\Soc(G)|$. Since $G$ has no nontrivial solvable normal
subgroup, $\Soc(G)$ is isomorphic to a direct product of non-abelian
simple groups, say
\[\Soc(G)\cong S_1\times S_2\times \cdots\times S_t.\]

First, we will show that $t$ is at most $k_{p'}(G)$. For each $1\leq
i\leq t$, choose a nontrivial $p$-regular element $x_i$ in $S_i$.
Consider the elements
\[g_i=(x_1,x_2,\cdots, x_i,1,\cdots,1)\in \Soc(G).\]
As $\ord(g_i)=\lcm(\ord(x_1),\ord(x_2),\cdots,\ord(x_i))$, these $t$
elements of $\Soc(G)$ are all $p$-regular. Moreover, since $G$
permutes the direct factors $S_1$, $S_2$,..., $S_t$ of $\Soc(G)$,
the elements $g_1$, $g_2$,..., $g_t$ belong to different conjugacy
classes of $G$. Hence, $G$ has at least $t$ classes of $p$-regular
elements so that $t\leq k_{p'}(G)$.

Next, we will show that each $|S_i|$ is bounded in terms of
$k_{p'}(G)$ for every $i$. Since there are finitely many sporadic
simple groups, we are left with two cases:

\medskip

(i) Suppose that $S_i=\Al_n$ is an alternating group. We will show
that $n$ is bounded in terms of $k_{p'}(G)$. Assume, to the
contrary, that $n$ can be arbitrarily large. We choose an $n$ large
enough such that the number of primes smaller than $n$ is at least
$k_{p'}(G)+2$. In that case, $|\Al_n|=n!/2$ is divisible by at least
$k_{p'}(G)+1$ primes different from $p$. For each such prime $q$,
there is at least one $G$-conjugacy class of $q$-elements.
Therefore, we deduce that $G$ would have more than $k_{p'}(G)$
classes of $p$-regular elements, a contradiction.

\medskip

(ii) Suppose that $S_i$ is a simple group of Lie type. Assume that
$|S_i|$ can be arbitrarily large. Then, by Proposition~\ref{Out
theorem}, the number of $\Aut(S_i)$-classes of $p$-regular elements
in $S_i$, denoted by $k^\ast_{p'}(S_i)$, would be arbitrarily large
as well. However, if $x$ and $y$ lie in different
$\Aut(S_i)$-classes of $S_i$, then $(1,...,1,x,1,...,1)$ and
$(1,...,1,y,1,...,1)$ lie in different $G$-classes, and we therefore
deduce that the number of $p$-regular classes of $G$ would be
arbitrarily large. We have shown that $|S_i|$ is bounded in terms of
$k_{p'}(G)$ and the proof is now complete.
\end{proof}

%%%%%%%%%%%%%%%%%%%%%%%%%%%%%%%%%%%%%%%%%%%%%%%%%%%%%%%%%%%%%%%%%%%%

\section{$p$-Singular classes of simple
groups}\label{section-singular-class-simple-groups}

In this section, we prove an analogue of Proposition~\ref{Out
theorem} for $p$-singular conjugacy classes. The following will be
needed in the proof of the first part of Theorem~\ref{main theorem
for p-singular classes}.

\begin{proposition}\label{proposition 2-singular class} Let $p$ be a fixed prime and $S$
a non-abelian simple group with $|S|$ divisible by $p$. Assume that
$S\neq \PSL_2(p^f)$. Assume furthermore that $S\neq \Sz(2^f)$ if
$p=2$ and $S\neq \ta G_2(3^f)$ if $p=3$. Let $k^\ast_p(S)$ denote
the number of $\Aut(S)$-classes of $p$-singular elements inside $S$.
Then we have
\[k^\ast_p(S)\rightarrow \infty \text{ as } |S|\rightarrow \infty.\]
\end{proposition}

\begin{proof} First we consider the alternating
groups of degree at least 5. By considering products of two distinct
cycles, one of which has length $p$, we see that
$k_p(\Al_n)\rightarrow \infty$ as $n\rightarrow \infty$. Moreover,
as $k^\ast_2(\Al_n)\geq k_2(\Al_n)/2$, it follows that
$k^\ast_2(\Al_n)\rightarrow \infty$ as $n\rightarrow \infty$ and the
proposition follows for the alternating groups. We therefore can
assume from now on that $G$ is a simple group of Lie type. Let
$h(S)$ denote the Coxeter number of the associated Weyl group of
$S$. Following~\cite{Isaacs-Kantor-Spaltenstein}, we denote by
$\mu(S)$ the probability that an element of $S$ is $p$-singular. In
other words, $\mu(S)$ is the proportion of $p$-singular elements in
$S$.

\medskip

(i) First, we assume that the characteristic of the underlying field
of $S$ is different from $p$.
By~\cite[Theorem~5.1]{Isaacs-Kantor-Spaltenstein}, we have
$\mu(S)\geq (1/h(S))(1-1/p)$ except when $p=3$ and $S=\PSL_3(q)$
with $(q-1)_3=3$ or $S=\PSU_3(q)$ with $(q+1)_3=3$, in which case
$\mu(S)=1/9$. (Here $x_3$ denotes the $3$-part of an integer $x$.)
So in any case, we always have
\[\mu(S)\geq \frac{p-1}{2ph(S)}>\frac{1}{4h(S)}.\]

The arguments for simple classical groups are fairly similar. We
present here a proof for the linear groups. So assume that $S\cong
\PSL_n(q)$. As mentioned in the proof of Lemma~\ref{lemma for linear
and unitary groups}, for $x\in\PSL_n(q)$, we have
\[|\Centralizer_{\PSL_n(q)}(x)|\geq\frac{q^{n-1}}{e(1+\log_q(n+1))(q-1,n)}.\] Therefore,
\begin{align*}k_p(\PSL_n(q))&\geq \frac{q^{n-1}}{4e(1+\log_q(n+1))(q-1,n)h(\PSL_n(q))}\\
&= \frac{q^{n-1}}{4en(1+\log_q(n+1))(q-1,n)}.
\end{align*}
Let $q=\ell^f$ where $\ell$ is a prime (unequal to $p$). Since
$|\Out(\PSL_n(q))|\leq 2f(n,q-1)$, we deduce that
\[k^\ast_2(\PSL_n(q))\geq
\frac{q^{n-1}}{8efn(1+\log_q(n+1))(q-1,n)^2}.\] It is now easy to
see that $k^\ast_2(\PSL_n(q))\rightarrow \infty$ as
$|\PSL_n(q)|\rightarrow\infty$.

Now we are left with the exceptional groups. Since $h(S)\leq 30$ for
every simple exceptional group $S$, we have $\mu(S)\geq 1/(120)$.
Assume that $S$ is defined over a field of $q$ elements with $r$ the
rank of the ambient algebraic group.
By~\cite[Theorem~6.15]{Fulman-Guralnick}, for every $x\in S$, we
have
\[|\Centralizer_S(x)|\geq
\frac{q^r}{A(\min\{q,r\})(1+\log_qr)}\geq\frac{q^r}{32A},\] where
$A$ is an absolute constant. It follows that
\[k_p(S)\geq \frac{q^r}{120\cdot 32A}=\frac{q^r}{3840A}.\] Thus
\[k_p^\ast(S)\geq \frac{q^r}{3840A\cdot 6f},\] where $q=\ell^f$ for
a prime $\ell\neq p$. Again, we see that $k^\ast_p(S)\rightarrow
\infty$ as $q\rightarrow \infty$.

\medskip

(ii) Next, we assume that the underlying field of $S$ has
characteristic $p$. By~\cite[Theorem
10.1]{Isaacs-Kantor-Spaltenstein}, we have
$$\mu(S)\geq \frac{2}{5q}.$$ (We note that the bound given in~\cite[Theorem
10.1]{Isaacs-Kantor-Spaltenstein} is $\frac{2}{5}q^{-\delta}$ where
$\delta=1$ unless $S$ is a Suzuki or Ree group, in which case
$\delta=2$. In that paper, the authors use notation $\ta B_2(q^2)$,
$\ta G_2(q^2)$, $\ta F_4(q^2)$ to denote the Suzuki and Ree groups.
Here we think it is more convenient to write $\ta B_2(q)$, $\ta
G_2(q)$, $\ta F_4(q)$.)

If $S$ is a simple classical group different from $\PSL_2(q)$, by
using the lower bound for the centralizer size as in (i), we can
also deduce that $k^\ast_p(S)\rightarrow \infty$ as $|S|\rightarrow
\infty$. The same thing is true for exceptional simple groups as
long as the rank $r$ associated to $S$ is greater than $1$ since
\[\frac{2}{5q}\cdot\frac{q^r}{32A\cdot 6f}\rightarrow \infty \text{ as } q\rightarrow \infty,\]
if $r>1$. Therefore, if $S$ is a simple group of exceptional Lie
type in characteristic $p$ with $S\neq \Sz(2^{f})$ when $p=2$ and
$S\neq \ta G_2(3^f)$ when $p=3$, then
\[k^\ast_p(S)\rightarrow\infty \text{ as } |S|\rightarrow\infty.\]
The proof is complete.
\end{proof}

%%%%%%%%%%%%%%%%%%%%%%%%%%%%%%%%%%%%%%%%%%%%%%%%%%%%%%%%%%%%%%%%%%%%%%%%%%%%

\section{First part of Theorem~\ref{main theorem for p-singular
classes}}\label{section6}

We are ready to prove the first part of Theorem~\ref{main theorem
for p-singular classes}, which we restate below for the reader's
convenience. The ideas are quite similar to those in
Theorem~\ref{main1}.

\begin{theorem}\label{main21} Let $p$ be a fixed prime, $F$ a positive real number, and $G$ a finite group whose non-abelian
composition factors have order divisible by $p$. Assume that $G$
does not contain $\PSL_2(p^f)$ with $f> F$ as a composition factor.
Assume furthermore that
\begin{enumerate}
\item[(i)] if $p=2$ then the Suzuki groups $\Sz(2^f)$ with $f>F$ are
not composition factors of $G$; and
\item[(ii)] if $p=3$ then the Ree groups $\ta G_2(3^f)$ with $f>F$ are
not composition factors of $G$.
\end{enumerate}
Then $|G/\bO_{\infty}(G)|$ is $(F,k_p(G))$-bounded.
\end{theorem}

\begin{proof} By Lemma~\ref{lemma k2(G/N) <k2(G)}, it suffices to assume that
$\bO_\infty(G)$ is trivial. Recall that $G$ does not contain
$\PSL_2(p^f)$ with $f>F$ as a composition factor. When $p=2$ or
$p=3$, the Suzuki groups $\Sz(2^f)$ or respectively the Ree groups
$\ta G_2(3^f)$ with $f>F$ are not composition factors of $G$. We
want to show that $|G|$ is bounded in terms of $F$ and $k_p(G)$. As
before, it is enough to bound $|\Soc(G)|$. Since $G$ has no
nontrivial solvable normal subgroup, $\Soc(G)$ is isomorphic to a
direct product of non-abelian simple groups, say $\Soc(G)\cong
S_1\times S_2\times \cdots\times S_t.$

First, we will show that $t$ is at most $k_p(G)$. We choose a
$p$-singular element $x_1$ in $S_1$ and arbitrary nontrivial
elements $x_2,x_3,...,x_t$ in $S_2,S_3,...,S_t$, respectively. (This
can be done since every non-abelian composition factor of $G$ has
order divisible by $p$.) Consider the elements
\[g_i=(x_1,x_2,\cdots, x_i,1,\cdots,1)\in \Soc(G)\] where
$i=1,2,...,t$. As
$\ord(g_i)=\lcm(\ord(x_1),\ord(x_2),\cdots,\ord(x_i))$, these $t$
elements of $\Soc(G)$ are $p$-singular. Moreover, the elements
$g_1$, $g_2$,..., $g_t$ belong to different conjugacy classes of
$G$. We conclude that $G$ has at least $t$ classes of $p$-singular
elements so that $t\leq k_p(G)$.

Now it suffices to show that if $S$ is a direct factor of $\Soc(G)$
different from $\PSL_2(p^f)$, and $\Sz(2^f)$ for $p=2$, and $\ta
G_2(3^f)$ for $p=3$, then $|S|$ is bounded in terms of $k_p(G)$.
Assume, to the contrary, that $|S_i|$ can be arbitrarily large.
Then, by Proposition~\ref{proposition 2-singular class}, the number
$k^\ast_p(S_i)$ of $\Aut(S)$-classes of $p$-singular elements in $S$
would be arbitrarily large as well. However, if $x$ and $y$ lie in
different $\Aut(S)$-classes of $S_i$, then $(1,...,1,x,1,...,1)$ and
$(1,...,1,y,1,...,1)$ lie in different $G$-classes, and we therefore
deduce that $k_p(G)$ would be arbitrarily large, as required.
\end{proof}

%%%%%%%%%%%%%%%%%%%%%%%%%%%%%%%%%%%%%%%%%%%%%%%%%%%%%%%%%%%%%%%%%%%%%%%%%%%%%%%%%%

\section{Proofs of Theorems~\ref{main theorem} and~\ref{main theorem for p-singular
classes}}\label{section-proofs of maintheorems}

We now can complete the proof of Theorems~\ref{main theorem}
and~\ref{main theorem for p-singular classes}. By
Theorems~\ref{main1}, \ref{main21} and Lemma~\ref{normal} we may
assume in this section that $G$ is solvable.

\begin{proof}[Proof of Theorem \ref{main theorem}]
Since $\bO_p(G)=1$, we have that $\bF(G)$ is a $p'$-group. Since
$\bF(G/\Phi(G))=\bF(G)/\Phi(G)$ and $k_{p'}(G/\Phi(G))\leq
k_{p'}(G)$, we may assume that $\Phi(G)=1$. By Gasch\"utz's theorem
(see, for instance, Theorem~1.12 of~\cite{mawo}), $H=G/\bF(G)$ acts
faithfully and completely reducibly on $V=\bF(G)$. Write
$V=V_1\oplus\cdots\oplus V_t$ as a direct sum of irreducible
$H$-modules. Now $H$ is isomorphic to a subgroup of the direct
product of the groups $G/\bC_G(V_i)$ for $i=1,\dots,t$. By
Corollary~\ref{c}, each of these quotient groups is metabelian by
$k$-bounded. Also, since $G$ has $k$ classes of $p$-regular
elements, we have $t\leq k$. Therefore, the direct product of the
groups $G/\bC_G(V_i)$ for $i=1,\dots,t$ is metabelian by
$k$-bounded. We deduce that $H$ is metabelian by $k$-bounded.
\end{proof}

The rest of the proof of Theorem~\ref{main theorem for p-singular
classes} goes along the same lines. In fact, as we already mentioned
in the introduction, we can obtain a bit more.

\begin{theorem}
Let $G$ be a solvable group with $\bO_{p'}(G)=1$. Then there exists
a normal subgroup $N$ containing $\bF(G)$ such that $N/\bF(G)$ is
metabelian and $|G/N|$ is bounded in terms of the number of classes
of $p$-elements of $G$.
\end{theorem}

\begin{proof}
Let $k$ be the number of conjugacy classes of $p$-elements of $G$.
Since $\bO_{p'}(G)=1$, the Fitting subgroup $\bF(G)$ is a $p$-group.
Since $\bF(G/\Phi(G))=\bF(G)/\Phi(G)$ and the number of classes of
$p$-elements of $G/\Phi(G)$ is at most the number of classes of
$p$-elements of $G$, we may assume that $\Phi(G)=1$. By Gasch\"utz's
theorem, $H=G/\bF(G)$ acts faithfully and completely reducibly on
$V=\bF(G)$. As in the proof above, we deduce that $H$ is metabelian
by $k$-bounded.
\end{proof}

%%%%%%%%%%%%%%%%%%%%%%%%%%%%%%%%%%%%%%%%%%%%%%%%%%%%%%%%%%%%%%

\section{Proofs of Theorems~\ref{main
theorem-p'-part} and~\ref{main
theorem-p-part}}\label{section-proof-Theorem-p'part}

We will prove Theorems~\ref{main theorem-p'-part} and~\ref{main
theorem-p-part} in this section. Since their proofs are partly
similar, we will skip some details in the proof of Theorem~\ref{main
theorem-p-part}.

\begin{proof}[Proof of Theorem~\ref{main theorem-p'-part}] In
light of Theorem~\ref{main1} and Lemma~\ref{normal} we may assume
that $G$ is solvable. We keep some of the notation in the proof of
Theorem~\ref{main theorem}. As before, we may assume that
$\Phi(G)=1$. Then Gasch\"utz's theorem implies that $H=G/\bF(G)$
acts faithfully and completely reducibly on $V=\bF(G)$. Write
\[V=V_1\oplus\cdots\oplus V_t\] as a direct sum of irreducible
$H$-modules. We have that $H$ is isomorphic to a subgroup of the
direct product of the groups $G/\bC_G(V_i)$ for $i=1,\dots,t$. Since
$G$ has $k$ classes of $p$-regular elements, we have $t\leq k$. We
write $G_i:=G/\bC_G(V_i)$ and notice that it suffices to bound
$|G_i|_{p'}$ for each $i=1,\dots,t$. We note that the number of
orbits of $G_i$ on $V_i$ is at most $k$. Therefore, by
Lemma~\ref{kel}, we may assume that $G_i$ is isomorphic to a
subgroup of the wreath product of $\Gamma(r^a)$ and $\Sy_l$, for
some $l$ that is $k$-bounded. Also, \[V_i=W_{i1}\oplus\cdots\oplus
W_{il},\] where $W_{ij}\cong\GF(r^a)$ for every $1\leq i\leq t$ and
$1\leq j\leq l$.

With a slight abuse of notation, we view $G_i$ as a subgroup of
$\Gamma(r^a)\wr \Sy_l$. In order to bound $|G_i|_{p'}$ it suffices
to bound $|G_i\cap\Gamma(r^a)^l|_{p'}$. We write
\[\Gamma(r^a)^l=\Gamma_1\times\cdots\times\Gamma_l.\] For each $1\leq
j\leq l$, we have $\Gamma_j=H_jB_j$ where $H_j$ is the Galois group
$\Gal(\GF(r^a)/\GF(r))$ and $B_j=\GF(r^a)^\ast$ and that $\Gamma_j$
acts on $W_{ij}$. We consider the homomorphisms \[\alpha_j:
G_i\cap\Gamma(r^a)^l\longrightarrow \Gamma_j\] defined by
\[\alpha_j(x_1,\dots,x_l)=x_j.\] In order to bound
$|G_i\cap\Gamma(r^a)^l|_{p'}$ it suffices to bound
$|\alpha_j(G_i\cap\Gamma(r^a)^l)|_{p'}$ for each $1\leq j\leq l$.

\medskip

\textbf{Claim 1}: $|\alpha_j(G_i\cap\Gamma(r^a)^l):\alpha_j(G_i\cap
B_1\times\cdots\times B_l)|_{p'}$ is $k$-bounded.

\medskip

Proof: Since $G_i$ is a quotient of $G$ and $k=k_{p'}(G)$, it
follows that $G_i$ has at most $k$ $p$-regular classes. The index of
$G_i \cap \Gamma(r^a)^l$ in $G_i$ is $k$-bounded, so the number of
$p$-regular classes of $G_i \cap \Gamma(r^a)^l$ is $k$-bounded.
Therefore, the number of $p$-regular classes of $\alpha_j(G_i \cap
\Gamma(r^a)^l)$ is $k$-bounded.

Now, we have that \[(G_i \cap \Gamma(r^a)^l)/(G_i\cap
B_1\times\cdots\times B_l)\] is abelian and since \[\alpha_j (G_i
\cap \Gamma(r^a)^l)/\alpha_j (G_i\cap B_1\times\cdots\times B_l)\]
is isomorphic to a quotient of this group, it is also abelian. By
the previous paragraph, the $p'$-part of its order is $k$-bounded
and the claim is proved.

\medskip

\textbf{Claim 2}: $|\alpha_j(G_i\cap B_1\times\cdots\times
B_l)|_{p'}/a$ is $k$-bounded.

\medskip

Proof: We observe that
$|\alpha_j(G_i\cap\Gamma(r^a)^l):\alpha_j(G_i\cap
B_1\times\cdots\times B_l)|$ divides $a$. Also, the group
$\alpha_j(G_i\cap B_1\times\cdots\times B_l)$ is abelian. Therefore,
\begin{align*}|\alpha_j(G_i\cap B_1\times\cdots\times B_l)|_{p'}&=k_{p'}(\alpha_j(G_i\cap B_1\times\cdots\times
B_l))\\
&\leq
k_{p'}(\alpha_j(G_i\cap\Gamma(r^a)^l))\left|\frac{\alpha_j(G_i\cap\Gamma(r^a)^l)}{\alpha_j(G_i\cap
B_1\times\cdots\times B_l)}\right|\\
&\leq ak_{p'}(\alpha_j(G_i\cap\Gamma(r^a)^l)).
\end{align*}
As argued above, the number of $p$-regular classes of $\alpha_j(G_i
\cap \Gamma(r^a)^l)$ is $k$-bounded and thus the claim follows.

Recall that the number of orbits of $G_i$ on $V_i$ is at most $k$.
Therefore the number of orbits of $G_i\cap\Gamma(r^a)^l$ on $V_i$ is
at most $k|G_i:G_i\cap\Gamma(r^a)^l|$, which is a $k$-bounded
quantity. We deduce that the number of orbits of
$\alpha_j(G_i\cap\Gamma(r^a)^l)$ on $W_{ij}$ is $k$-bounded for
every $j$. Therefore, \[(r^a-1)/|\alpha_j(G_i\cap\Gamma(r^a)^l)|
\text{ is } k\text{-bounded}.\] In other words, we have
\[\frac{(r^a-1)_{p'}}{|\alpha_j(G_i\cap\Gamma(r^a)^l)|_{p'}}\cdot \frac{(r^a-1)_{p}}{|\alpha_j(G_i\cap\Gamma(r^a)^l)|_{p}} \text{ is } k\text{-bounded}.\]
Since $\alpha_j(G_i\cap\Gamma(r^a)^l)$ is a subgroup of $\Gamma_j$
whose order is $a(r^a-1)$, it follows that
\[|\alpha_j(G_i\cap\Gamma(r^a)^l)|_{p}\leq a_p(r^a-1)_p.\] We
therefore find that
\[\frac{(r^a-1)_{p'}}{a_p\cdot|\alpha_j(G_i\cap\Gamma(r^a)^l)|_{p'}} \text{ is } k\text{-bounded}.\]
Combining this with the claims above, we have
\[\frac{(r^a-1)_{p'}}{a\cdot a_p} \text{ is } k\text{-bounded}.\]

We now apply Lemma~\ref{lemma-numbertheory} to deduce that $a$ is
bounded in terms of $k$. For each $1\leq j\leq l$, as
$|\alpha_j(G_i\cap\Gamma(r^a)^l)|_{p'}/a$ is $k$-bounded, we
conclude that $|\alpha_j(G_i\cap\Gamma(r^a)^l)|_{p'}$ is also
bounded in terms of $k$. The proof is complete.
\end{proof}

%-------------------------

\begin{proof}[Proof of Theorem~\ref{main theorem-p-part}] By Theorem~\ref{main21} and Lemma~\ref{normal} we may assume
that $G$ is solvable. Let $k:=k_p(G)$ be the number of $p$-singular
classes of $G$. As before, we may assume that $\Phi(G)=1$. Then
Gasch\"utz's theorem implies that $G/\bF(G)$ acts faithfully and
completely reducibly on $\bF(G)$, which is a $p$-group since
$\bO_{p'}(G)=1$.

We now follow the notation in the proof of Theorem~\ref{main
theorem-p'-part}. In particular, we define $V_i$, $G_i$, $W_{ij}$,
$\alpha_j$, and $\Gamma_j=H_jB_j$ with $1\leq i\leq t$ and $1\leq
j\leq l$ as before. In order to bound $|G/\bF(G)|_{p}$, it suffices
to bound $|\alpha_j(G_i\cap\Gamma(r^a)^l)|_{p}$ for each $i$ and
$j$.

As in the proof of Theorem~\ref{main theorem-p'-part}, we can argue
that \[k_p(\alpha_j(G_i\cap\Gamma(r^a)^l))\] and
\[|\alpha_j(G_i\cap\Gamma(r^a)^l):\alpha_j(G_i\cap
B_1\times\cdots\times B_l)|_{p}\] are both $k$-bounded. Furthermore,
since \[|\alpha_j(G_i\cap\Gamma(r^a)^l):\alpha_j(G_i\cap
B_1\times\cdots\times B_l)|\] divides $a$, we have
\[k_p(\alpha_j(G_i\cap
B_1\times\cdots\times B_l))\leq a
k_p(\alpha_j(G_i\cap\Gamma(r^a)^l)).\] Since the group
$\alpha_j(G_i\cap B_1\times\cdots\times B_l)$ is abelian, we see that
\[k_p(\alpha_j(G_i\cap B_1\times\cdots\times B_l))=|\alpha_j(G_i\cap B_1\times\cdots\times B_l)|_{p'}(|\alpha_j(G_i\cap B_1\times\cdots\times
B_l)|_p-1).\] As $k_p(\alpha_j(G_i\cap\Gamma(r^a)^l))$ is
$k$-bounded, it follows that
\[\frac{1}{a}|\alpha_j(G_i\cap B_1\times\cdots\times B_l)|_{p'}(|\alpha_j(G_i\cap B_1\times\cdots\times
B_l)|_p-1) \text{ is } k\text{-bounded}.\]

If $|\alpha_j(G_i\cap B_1\times\cdot\times B_l)|_p=1$ then
\[|\alpha_j(G_i\cap\Gamma(r^a)^l)|_p=|\alpha_j(G_i\cap\Gamma(r^a)^l):\alpha_j(G_i\cap
B_1\times\cdots\times B_l)|_{p}\] is $k$-bounded and we are done. So
we may assume that \[|\alpha_j(G_i\cap B_1\times\cdot\times
B_l)|_p\geq 2\] so that \[|\alpha_j(G_i\cap B_1\times\cdots\times
B_l)|_p-1\geq \frac{1}{2}|\alpha_j(G_i\cap B_1\times\cdots\times
B_l)|_p.\] We then deduce that
\[\frac{|\alpha_j(G_i\cap B_1\times\cdots\times B_l)|}{a} \text{ is } k\text{-bounded}.\]

Similar to the proof of Theorem~\ref{main theorem-p'-part}, we know
that \[(r^a-1)/|\alpha_j(G_i\cap\Gamma(r^a)^l)| \text{ is }
k\text{-bounded}.\] Using the conclusion of the previous paragraph
and the fact that \[|\alpha_j(G_i\cap\Gamma(r^a)^l):\alpha_j(G_i\cap
B_1\times\cdots\times B_l)| \text{ divides } a,\] we deduce that
\[\frac{r^a-1}{a^2} \text{ is } k\text{-bounded}.\]
Therefore, the integer $a$ must be $k$-bounded, which in turn
implies that \[|\alpha_j(G_i\cap B_1\times\cdots\times B_l)|_p
\text{ is } k\text{-bounded}.\] As mentioned above that
\[|\alpha_j(G_i\cap\Gamma(r^a)^l):\alpha_j(G_i\cap
B_1\times\cdots\times B_l)|_{p} \text{ is } k\text{-bounded},\] we
conclude that \[|\alpha_j(G_i\cap\Gamma(r^a)^l)|_{p} \text{ is }
k\text{-bounded},\] as required.
\end{proof}

%%%%%%%%%%%%%%%%%%%%%%%%%%%%%%%%%%%%%%%%%%%%%%%%%%%%%%%%%%%%%%%%

\section{Examples and open questions}\label{section-example}

There does not seem to be much room for improvement in Theorem
\ref{main theorem}. For instance, let $V_{l,q}=\GF(q^{2^l})$ where
$q$ is any Fermat prime and $l$ any positive integer, and
$H_{l,q}=C_{q-1}\wr P$, where $P$ is a Sylow $2$-subgroup of
$\Sy_{2^l}$. Take $G_{l,q}:=H_{l,q}V_{l,q}$. We see that the number
of $2$-regular classes of $G$ is $2^l+1$ (and a complete set of
representatives of these classes is $(1,\dots,1), (a,1,\dots,1),
(a,a,1,\dots,1),\dots,(a,\dots,a)$). If we fix $l$ then all the
groups $G_{l,q}$ have the same number of $2$-regular classes. It
turns out that the order of the groups $G_{l,q}$ is bounded if and
only if there are finitely many Fermat primes. On the other hand, if
we fix $q$ and take $l$ large enough then $G_{l,q}$ has arbitrarily
large derived length. Therefore, there is no hope to prove that  if
$G$ is any finite group with $\bO_p(G)=1$ and $k=k_{p'}(G)$ then
either $G$ is metabelian or $|G/\bF(G)|$ is $k$-bounded, for
instance (unless we prove that there are finitely many Fermat primes
first).

We have already seen that if we want to bound the order of a group
$G$ with $\bO_p(G)=1$ in terms of the number of $p$-regular classes
we need to bound the number of Fermat primes. In fact, looking at
the classification of groups with few $p$-regular classes we see the
following.

\begin{theorem} We have the following.
\begin{enumerate}
\item[(i)]
The cardinality of the set of groups with two $p$-regular classes
and no nontrivial normal $p$-group (for some prime $p$) is bounded
if and only if there are finitely many Fermat and Mersenne primes.
\item[(ii)]
The cardinality of the set of groups with three $p$-regular classes
and no nontrivial normal $p$-group (for some prime $p$) is bounded
if and only if there are finitely many Fermat primes, finitely many
primes of the form $2r^n+1$ where $r$ is prime, and finitely many
$3$-powers of the form $2r^n+1$ where $r$ is prime.
\item[(iii)]
The cardinality of the set of groups with four $p$-regular classes
and no nontrivial normal $p$-group for some odd prime $p$ is bounded
if and only if there are finitely many Mersenne primes and finitely
many prime powers of the form $4r^n+1$, where $r$ is prime.
\end{enumerate}
\end{theorem}

\begin{proof}
Part (i) follows from the classification of groups with
$k_{p'}(G)=2$ in~\cite{nw}. Part~(ii) follows from the
classification of groups with $k_{p'}(G)=3$ in~\cite{n3,n1,n2}.
Part~(iii) follows from~\cite{tie}, noting that in the third family
of groups that appear in the statement of Main Theorem it should say
that $p=2^n-1$ is a Mersenne prime.
\end{proof}

We remark that, in fact, using the groups that appear in Ninomiya's
and Tiedt's theorems it is easy to build up groups similar to the
groups $G_{l,q}$ in the first paragraph of this section for odd
primes.

In view of these results and the proof of Theorem~\ref{main theorem}
we suspect that it should be possible to prove that the cardinality
of the set of groups $G$ with $k$ $p$-regular classes and
$\bO_p(G)=1$ is $k$-bounded if and only if there are finitely many
primes of certain forms. Our proof basically shows that it suffices
to study actions of affine linear groups on modules and see what
happens when there are few orbits.

As mentioned in the introduction, Landau's theorem was strengthened
by Craven and Moret\'{o} by showing that the order of a finite group
$G$ is bounded in terms of the largest multiplicity of irreducible
character degrees of $G$. We think that Theorem~\ref{main theorem}
can also be strengthened in this direction.

\begin{question} Let $p$ be a fixed prime and $G$ a finite group
with $\bO_p(G)=1$. Let $m$ be the largest multiplicity of $p$-Brauer
character degrees or the largest multiplicity of $p$-regular class
sizes of $G$. Is it true that $|G/\bO_\infty(G)|$ is $m$-bounded and
$\bO_\infty(G)/\bF(G)$ is metabelian by $m$-bounded?
\end{question}

Another direction to generalise Theorems~\ref{main theorem}
and~\ref{main theorem-p'-part} is to go from a single prime $p$ to a
set of primes $\pi$.

\begin{question}\label{question} Let $\pi$ be a fixed finite set of primes of cardinality at most $2$ and $G$ a finite group. Is it true that $|G/\bO_\infty(G)|$ is
bounded in terms of the number of conjugacy classes of $\pi$-regular
elements of $G$?
\end{question}

\noindent This question is related to a recent study
\cite{Maroti-Ng} on the number of conjugacy classes of $\pi$-regular
elements. It has been shown there that this number somehow controls
the $\pi$-local structure of the group. The answer for
Question~\ref{question} is negative if the cardinality of $\pi$ is
$3$ or higher. For instance, a direct product of any number of
copies of $\Al_5$ has no nontrivial $\{2,3,5\}$-regular classes but
its order can be arbitrarily large. However, as long as the prime
$2$ is not in $\pi$, we believe that the same conclusion still
holds.

\begin{question} Let $\pi$ be a fixed finite set of primes such that $2\notin \pi$ and $G$ a finite group. Is it true that $|G/\bO_\infty(G)|$ is
bounded in terms of the number of conjugacy classes of $\pi$-regular
elements of $G$?
\end{question}

%%%%%%%%%%%%%%%%%%%%%%%%%%%%%%%%%%%%%%%%%%%%%%%%%%%%%%%%%%%%%%%

\section{Groups with four $p$-regular classes - Theorem \ref{B}}\label{section-proof-Theorem-B}

It would be interesting to classify the nonsolvable groups with
$\bO_p(G)$=1 and at most five $p$-regular classes. The reason for
this is that we expect the number of these groups to be finite.
However, even if the number of nonsolvable groups with six
$p$-regular classes and $\bO_p(G)=1$ were finite, proving this would
be out of reach. The reason is that, for instance, the direct
products $\Sy_5\times F_n$, where the $F_n$ are the Frobenius groups
defined in the introduction, have six $2$-regular classes.

We now start working toward a proof of Theorem~\ref{B}. We will use
Ninomiya's classification of nonsolvable groups with three
$p$-regular classes.

\begin{lemma}[Ninomiya \cite{n3}]\label{list of nonsolvable groups with kp(G)=3} Let $G$
be a finite non-solvable group with $\bO_p(G)=1$. Then $k_{p'}(G)=3$
if and only if one of the following holds:
\begin{enumerate}
\item[(i)] $p=2$ and $G\cong \Sy_5$, $\PSL_2(7)\cdot 2$, $\Al_6\cdot
2_3$, or $\Al_6\cdot 2^2$.
\item[(ii)] $p=3$ and $G\cong \PSL_2(8)\cdot 3$.
\item[(iii)] $p=5$ and $G\cong \Al_5$.
\end{enumerate}
\end{lemma}

In the next two results we also consider groups with five
$p$-regular classes. This is not necessary for the proof of
Theorem~\ref{B}, but we have decided to include it because it does
not make the proofs much longer and it would be helpful for the
classification on nonsolvable groups with five $p$-regular classes
and trivial solvable radical.

\begin{lemma}\label{simple groups with 4 prime divisors} Let $S$ be
a simple group whose order has exactly four distinct prime divisors.
Let $G$ be an almost simple group with socle $S$. Then

\begin{enumerate}
\item[(i)] $G$ has exactly four $p$-regular classes if and only if $p=2$
and $G\cong \PSL_3(4)\cdot 2_1$, $\PSL_3(4)\cdot 2^2$, or $\ta
F_4(2)'\cdot 2$.

\item[(ii)] $G$ has exactly five $p$-regular classes if and only if $p=2$
and $G\cong \Sy_7$, $M_{11}$, $M_{12}\cdot 2$, $\PSL_3(4)\cdot 2_2$,
$\PSL_3(4)\cdot 2_3$, $\ta F_4(2)'$, $\PSL_2(11)\cdot 2$.
\end{enumerate}
\end{lemma}

\begin{proof} It is not known whether the collection of finite simple groups whose order has exactly four prime
divisors is finite or not. However, Y.~Bugeaud, Z.~Cao, and
M.~Mignotte showed in~\cite{Bugeaud-Cao-Mignotte} that if $S$ is
such a group, then $S\cong \PSL_2(q)$ where $q$ is a prime power
satisfying
\begin{equation}\label{equation 4 prime divisors}q(q^2-1)=\gcd(2,q-1)2^{\alpha_1}3^{\alpha_2}r^{\alpha_3}s^{\alpha_4}\end{equation}
with $r$ and $s$ prime numbers such that $3<r<s$, or $S$ belongs to
a finite list of `small' simple groups,
see~\cite[Theorem~1]{Bugeaud-Cao-Mignotte} for the list.
Using~\cite{Atl1,Atl2}, it is routine to check the list and find all
the groups with exactly four or five $p$-regular classes.

So we are left with the case $S\cong \PSL_2(q)$ mentioned above.
By~\cite[Theorem~6.4]{Fulman-Guralnick}, the centralizer size of an
element of $\GL_2(q)$ can be estimated by
\[|\Centralizer_{\GL_2(q)}(g)|\geq \frac{q(q-1)}{e(1+\log_q3)}.\] Therefore, the centralizer size of an element of $\PSL_2(q)$ is at least
\[\frac{q(q-1)}{e(1+\log_q3)(q-1)(q-1,2)}=\frac{q}{e(1+\log_q3)(q-1,2)}.\]
On the other hand, by~\cite[Theorem~1.1]{Babai-Praeger-Wilson}, the
proportion of $p$-regular elements of $\PSL_2(q)$ is greater than or
equal to $1/4$. In particular, the number of $p$-regular elements in
$\PSL_2(q)$ is at least $({1}/{4})|\PSL_n(q)|$ and it follows that
\[k_{p'}(\PSL_2(q))\geq\frac{q}{4e(1+\log_q3)(2,q-1)},\] which in turns implies
\[5\geq k_{p'}(G)\geq \frac{q}{4ef(1+\log_q3)(2,q-1)^2},\] where
$q=\ell^f$ for a prime $\ell$. This inequality together with the
partial solution of Equation~\ref{equation 4 prime divisors}
in~\cite{Bugeaud-Cao-Mignotte} yield
\[q\in\{2^4,2^5,2^7,3^3,3^4,5^2,7^2,11,13,19,23,31,37,47,53,73,97\}.\]
Using~\cite{Atl2,GAP4}, we have checked that there is no such group
with exactly four $p$-regular classes and only one group with
exactly five $p$-regular classes and this group is $\PSL_2(11)\cdot
2$.
\end{proof}

\begin{lemma}\label{lemma G is almost simple} Let $G$ be a finite group with $\bO_\infty(G)=1$. If $k_{p'}(G)\leq 5$, then
$G$ is almost simple.
\end{lemma}

\begin{proof} Let $M$ be a minimal normal subgroup of $G$. Then $M$
is a direct product of $n$ copies of a non-abelian simple group, say
\[M\cong \underbrace{S\times S\times\cdots \times S}_{n \text{ times }}.\]
First we prove that $n=1$. Assume the contrary that $n\geq2$. Since
$|S|$ has at least three prime divisors, we can choose two elements
$x_1$ and $x_2$ in $S$ of prime order different from $p$ and
$\ord(x_1)\neq \ord(x_2)$. Consider the following $p$-regular
elements:
\[(x_1,1,1,\cdots,1), (x_1,x_1,1,\cdots,1), (x_2,1,1,\cdots,1),
(x_2,x_2,1,\cdots,1), (x_1,x_2,1,...,1).\] It is clear that these
elements belong to different classes of $G$. Therefore, $G$ has at
least six classes of $p$-regular elements (including the trivial
class), a contradiction. So $n=1$ or in other words, $M=S$ is
simple.

Next we prove that $\Centralizer_G(S)=1$. Again assume the contrary
that $\Centralizer_G(S)\neq 1$. Since $\Centralizer_G(S)\lhd G$ and
$\bO_\infty(G)=1$, $\Centralizer_G(S)$ is non-solvable and hence
$|\Centralizer_G(S)|$ has at least three prime divisors. As above,
we can choose two elements $y_1$ and $y_2$ in $\Centralizer_G(S)$ of
prime order different from $p$ and $\ord(y_1)\neq \ord(y_2)$. Then
the elements \[1,x_1, x_2, y_1, y_2, x_1 y_1\] belong to different
classes of $G$ and again we have a contradiction. Thus we have shown
that $\bC_G(S)=1$ so that $G\leq \Aut(S)$. So $G$ is almost simple
with socle $S$.
\end{proof}

Now we are ready to prove Theorem~\ref{B}.

\begin{proof}[Proof of Theorem~\ref{B}] By Lemma~\ref{lemma G is almost simple}, we know that
$G$ is almost simple with socle $S$. Since $G$ has exactly four
$p$-regular classes, the number of prime divisors of $|S|$ is at
most $4$. First, we assume that the number of prime divisors of
$|S|$ is $3$. There are only $8$ such simple groups
(see~\cite{Herzog} or~\cite[p.~12]{Gorenstein1}) and we have
\[S\in\{\Al_5,\PSL_2(7),\Al_6, \PSL_2(8), \PSL_2(17), \PSL_3(3),\PSU_3(3), \PSU_4(2)\}.\]
Using~\cite{Atl2}, one can easily find all almost simple groups with
socle $S$ in this list and exactly four $p$-regular classes.

Now we are left with the case where $|S|$ has exactly four prime
divisors. But this is done in Lemma~\ref{simple groups with 4 prime
divisors}(i). The proof is complete.
\end{proof}

\section*{Acknowledgement} The authors are grateful to the referee
for several helpful comments that have significantly improved the
exposition of the paper.

%%%%%%%%%%%%%%%%%%%%%%%%%%%%%%%%%%%%%%%%%%%%%%%%%%%%%%%%%%%%%%%%%%%%%%%%%%%%%

\end{document}